\documentclass[12pt,a4paper]{article}
\usepackage[centertags]{amsmath}
\usepackage{amsfonts}

\usepackage{graphics}
\usepackage{amssymb}
\usepackage{amsthm}
\usepackage{newlfont}
\usepackage{epsfig}
\usepackage{amsmath}
\usepackage{graphicx}
\usepackage{makeidx}
\usepackage{mathtools}
\usepackage{listings}
\usepackage{tikz}

\theoremstyle{plain}
\newtheorem{thm}{Theorem}[section]

\newtheorem{lem}[thm]{Lemma}

\theoremstyle{definition}

\theoremstyle{remark} \tolerance=10000 \hbadness=10000
\lstdefinestyle{mystyle}{
    basicstyle=\ttfamily\tiny,
    breakatwhitespace=false,
    breaklines=true,
}

\def \ni{\noindent}
\title{\bf On Geodesic Leech Labeling of Some Graph Classes}
\author {Aparna Lakshmanan S.\footnote{E-mail : aparnals@cusat.ac.in, aparnaren@gmail.com} ~and Arun J. Manattu\footnote{arunjmanattu@gmail.com}\\
	Department of Mathematics\\
	Cochin University of Science and Technology\\
    Cochin - 22, Kerala, India.
}
\date{}
\begin{document}
\maketitle
\begin{abstract}
\ni\line(1,0){360}\\
Let $f:E\rightarrow \{1,2,3,\dots\}$ be an edge labeling of $G$. The geodesic path number of $G$, $t_{gp}(G)$, is the number of geodesic paths in $G$. An edge labeling $f$ is called a geodesic Leech labeling, if the set of weights of the geodesic paths in $G$ is $\{1,2,3,\dots,t_{gp}(G)\}$, where the weight of a path $P$ is the sum of the labels assigned to the edges of $P$. A graph which admits a geodesic Leech labeling is called a geodesic Leech graph. Otherwise, we call it a non-geodesic Leech graph. In this paper, we prove that cycles $C_n$, $n \geq 5$ are non-geodesic Leech graphs. We also prove that there are at most three regular complete bipartite graphs that are geodesic Leech. We show that degree sequence cannot characterize geodesic Leech graphs. The geodesic path number of the wheel graph $W_n$ is obtained and the geodesic Leech labeling of $W_5$ and $W_6$ is given.
\ni\line(1,0){360}\\
\ni {\bf Keywords:}  Leech Labeling, Geodesic path number, Geodesic Leech Labeling, Geodesic Leech Graph\\

\ni {\bf AMS Subject Classification:} {\bf 05C78}\\
\ni\line(1,0){360}\\
\end{abstract}

\section{Introduction}
By a graph $G=(V,E)$ we mean a finite simple undirected graph. The order $\left|V\right|$ and
the size $\left|E\right|$ of $G$ are denoted by $n$ and $m$ respectively.
For all graph theoretic terminology and notations not mentioned here, we refer to Balakrishnan and Ranganathan
\cite{Bal}.\\

Let $f:E\rightarrow \{1,2,3,\dots\}$ be an edge labeling of $G$. The
weight of a path $P$ in $G$ is the sum of labels of the edges
of $P$ and is denoted by $w(P)$. In 1975, Leech \cite{Lee} introduced the
concept of Leech trees, utilizing the unique property of a tree that there exists a unique path between every pair of vertices.\\

Let $T$ be a tree of order $n$. An edge labeling $f:E\rightarrow
\{1,2,3,\dots\}$ is called a Leech labeling if the weights of the
$^nC_2$ paths in $T$ are exactly 1, 2,\dots,$^nC_2$. A tree which
admits a Leech labeling is called a Leech tree. Since each edge
label is the weight of a path of length one, it follows that $f$
is an injection and 1, 2 are edge labels for all $n \geq 3$. Leech
gave examples of five Leech trees in \cite{Lee} itself and these are
the only known Leech trees till date. Apart from the different negative results about the existence of Leech trees \cite{Sze} \cite{See5} \cite{See6}, a positive result that stands out is from Herbert Taylor \cite{Tay} which states that a tree with $n$ vertices is Leech only if $n$ is either a perfect square or two added to a perfect square.\\

Some variations to Leech labeling have since been studied at various junctures of time. In 2007, Bill Calhoun and his team published a paper on the notion of minimal distinct distance trees\cite{Bil}. David Leach worked extensively on the idea of Generalized Leech trees \cite{Lea} and also made contributions towards the notion of Modular Leech trees\cite{Lea2}. As recently as in 2022, the parameter Leech index of a tree was introduced\cite{See3}, which denoted how close a tree is towards being Leech.\\

Though the concept of Leech labeling was motivated by a problem in electrical networks, where $n-1$ resistors can be used to obtain all resistances from $1$ to $^nC_2$, as there are only 5 Leech trees identified till date, the Leech labeling became almost useless. Keeping this in mind, in \cite{See}, the concept of Leech labeling was extended to the class of all graphs in two different ways, each of which is interesting for various reasons. In this paper, we focus on one of those extensions, namely geodesic Leech labeling.\\

The path number, $t_{p}(G)$ of a graph $G$ is the total number of paths in a graph $G$ \cite{See}. Let	$f:E\rightarrow \{1,2,3,\dots\}$ be an edge labeling of $G$. If the set of weights of the  $t_{p}(G)$  paths in $G$ is $\{1,2,3,\dots, t_{p}(G)\}$, then $f$ is called a Leech labeling of $G$ and a graph which admits a Leech labeling is called a Leech graph \cite{See}. There are only very few published works exploring Leech Labeling of graphs. \cite{See} \cite{See4} \cite{Mer}. Numerous problems are still wide open in this area for potential research. \\'

The geodesic path number, $t_{gp}(G)$ of a graph $G$ is the total number of geodesic paths in a graph $G$ \cite{See}. Let	$f:E\rightarrow \{1,2,3,\dots\}$ be an edge labeling of $G$. If the set of weights of the  $t_{gp}(G)$ geodesic paths in $G$ is $\{1,2,3,\dots, t_{gp}(G)\}$, then $f$ is called a geodesic Leech labeling of $G$ and a graph which admits a geodesic Leech labeling is called a geodesic Leech graph \cite{See}. A graph that does not admit such a labeling is called a non-geodesic Leech graph. The value of $t_{gp}(G)$ for various families of graphs is given in \cite{See} and it is observed that cycles of length 3 and 4, $K_n$ and $K_n - e$, where $e$ is any edge, for every $n$ are geodesic Leech graphs. In \cite{See2}, four infinite families of geodesic Leech graphs are given and it is observed that $C_5$ is not a geodesic Leech graph. It was left as an open problem to prove that $C_n$'s and $K_{m,n}$'s are non-geodesic Leech graphs except for finitely many.\\

The paper \cite{See2} concludes by defining almost geodesic Leech graphs gaining inspiration from \cite{Oze} in which Ozen et al. defined almost Leech trees by introducing a minor relaxation to the condition for being a Leech tree. Let	$f:E\rightarrow \{1,2,3,\dots\}$ be an edge labeling of $G$. If the set of weights of the  $t_{gp}(G)$ geodesic paths in $G$ is missing exactly one element from $\{1,2,3,\dots, t_{gp}(G)\}$ and one path weight is repeated, then $f$ is called an almost geodesic Leech labeling of $G$ and a graph which admits an almost geodesic Leech labeling is called an almost geodesic Leech graph.\\

In this paper, we prove that cycles of length greater than 4 are non-geodesic Leech graphs, settling the open problem given in \cite{See2}. We also prove that there are at most 3 regular complete bipartite graphs that are geodesic leech. We also show that degree sequence doesn't characterize geodesic Leech graphs using the two cubic graphs of six vertices. The geodesic path number of wheel graph $W_n$ is obtained and the geodesic Leech labeling of $W_5$ and $W_6$ is also given. We believe that $W_n$ for $n \geq 7$ are non-geodesic Leech graphs and this is left as an open problem. \\
\par For the reader to have a feel for the notion of geodesic Leech labeling, we have given the figures of some geodesic Leech graphs and their labeling in Appendix II. One could observe that neither geodesic Leech labeling nor almost geodesic Leech labeling of a graph need to be unique. We have also shown the labeling of some almost geodesic Leech graphs. In fact, even though our results prove the existence of infinite families of non geodesic Leech graphs, there are very many graphs with the property of being geodesic Leech as already shown in the paper \cite{See2}. 8 out of the 9 forbidden subgraphs of line graphs are geodesic Leech (See \ref{Figure 4}). We also observed that graphs with at most 5 vertices are either Geodesic Leech or almost Geodesic Leech. The edge labelings that validate our observations have been provided in Appendix II. \\

\section{Edge-Transitive Graphs}

\begin{lem}
    Let G be an edge transitive graph with each edge labelled $a_i$ appearing in exactly $k$ geodesic paths and let $f$  be a Geodesic Leech labeling of G, then
    \begin{equation*}
        k\sum_{i=1}^na_i = \frac{t_{gp}(G)(t_{gp}(G)+1)}{2}.
    \end{equation*}
\end{lem}

\par This equality arises from the computation of geodesic path weights in two different ways since every edge in an edge transitive graph appears in equal number of geodesic paths. \\

The following lemma gives the value of $t_{gp}(C_n)$.

\begin{lem} \label{tgp} \cite{See}
	Let $C_n = (v_1,v_2,...,v_{n}, v_1)$. Then,\\ \[ t_{gp}(C_n) =
\begin{dcases}
      k(2k+1) & if\ n = 2k+1 \\
      2k^2 & if\ n = 2k \\
     \end{dcases}
\]
\end{lem}

The geodesic Leech labeling of cycles of length 3 and 4 is given in Figure 1. In this section, we prove that all the remaining cycles are non-geodesic Leech graphs. Throughout this section, let $C_n$ be the cycle with vertex set $\{v_1,v_2,\dots,v_n\}$ and $v_i$ adjacent to $v_{i+1}$ for $i = 1,2,\dots,n$, let $f$ be a geodesic Leech labeling of $C_n$ and let $f(v_iv_{i+1}) = a_i$, for $i = 1,2,\dots,n$ where addition of the indices is taken modulo $n$.

 \begin{figure}[h]\center \label{fig3}
	 	\includegraphics[width=8cm]{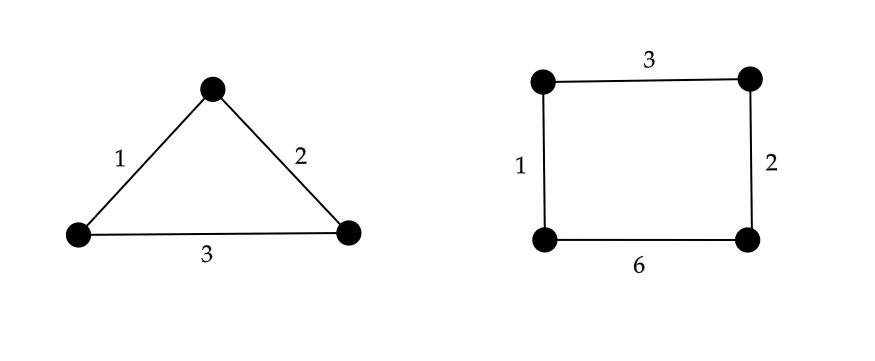}\\[-0.3cm]
	 	\caption{Geodesic Leech labeling of $C_3$ and $C_4$}
	 \end{figure}

\begin{lem} \label{lemma1}
  If $C_n$ is a geodesic Leech cycle, then $n = 3, 4$ or $10$.
\end{lem}
\begin{proof}
   We count the sum of weights of all geodesic paths in two different ways. Note that the number of geodesic paths of length $l$ containing a fixed edge $e$ is $l$. Therefore, the total number of geodesic paths containing an edge $e$ is $1+2+\dots+diam(G) = \frac{d(d+1)}{2}$, where $d = diam(G) = \lfloor\frac{n}{2}\rfloor$. Hence, the sum of weights of all geodesic paths will be $\frac{d(d+1)}{2}\sum_{i=1}^na_i$.\\

  Since $f$ is a geodesic Leech labeling, the weights of geodesic paths are precisely 1, 2, $\dots, t_{gp}(G)$. Therefore, the sum of weights of all geodesic paths must be $\frac{t_{gp}(G)(t_{gp}(G)+1)}{2}$. Hence, $\frac{d(d+1)}{2}\sum_{i=1}^na_i = \frac{t_{gp}(G)(t_{gp}(G)+1)}{2}$.\\

\ni{\bf Case 1:} $n = 2k+1$\\

In this case, $t_{gp}(G) = 2k^2+k$ and $d=k$. Therefore, we have $\frac{k(k+1)}{2}\sum_{i=1}^na_i = \frac{(2k^2+k)(2k^2+k+1)}{2}$. Since, $\sum_{i=1}^na_i$ is a natural number, $\frac{k(k+1)}{2}$ must divide $\frac{(2k^2+k)(2k^2+k+1)}{2}$. On simplification, we can see that this is possible only for $k=1$, which in turn gives $n=3$.\\

\ni{\bf Case 2:} $n = 2k$\\

In this case, $t_{gp}(G) = 2k^2$ and $d=k$, so that $\frac{k(k+1)}{2}$ must divide $\frac{2k^2(2k^2+1)}{2}$. On simplification, we get $k =$ 1, 2 or 5. Since, $n \geq 3$, we have, $n =$ 4 or 10.\\

Hence the lemma follows.
\end{proof}

Figure 1 gives the geodesic Leech labeling of $C_3$ and $C_4$ and hence the only question left is whether $C_{10}$ is a geodesic Leech graph or not? We have written a python programme (given as Appendix I in this paper), which confirms that $C_{10}$ do not admit any geodesic Leech labeling. To write the programme, we have used the fact that, if a geodesic Leech labeling of $C_{10}$ exists, then the maximum edge label cannot exceed 31.\\

\begin{lem}\label{lemma2}
  If a geodesic Leech labeling of $C_{10}$ exists, then the maximum edge label cannot exceed 31.
\end{lem}
If $f$ is a geodesic Leech labeling of $C_{10}$, then it follows from the proof of Lemma \ref{lemma1} that $w(f) = \sum_{i=1}^{10}a_i = 85$. Also, from Lemma \ref{tgp}, $t_{gp}(C_{10}) = 50$. Also, if a path of length five has weight $w$, then its complement path must have weight $85 - w$. Therefore, the path weight of a path of length five must be at least $85 - 50 = 35$. Now, every edge $e$ belongs to 15 geodesic paths and there are five geodesic paths of length five which do not contain the edge $e$. Let $f$ be a geodesic Leech labeling of $C_{10}$, if it exists, and let $e$ be the edge with maximum label $M$. Then, all the 15 geodesic paths which contains the edge $e$ must have path weight greater than or equal to $M$ and the five geodesic paths of length 5, which do not contain the edge $e$ also must have path weight greater than or equal to 35. Since, the maximum geodesic path weight is 50, it follows that $M$ is at most 31.

\begin{thm}\label{cycle}
  The cycles $C_n$ are non-geodesic Leech graphs, except for $n = 3$ and 4.
\end{thm}

\begin{lem} \cite{See}
    $t_{gp}(K_{n,n}) = n^3  (n^2$ edges and $n^3 - n^2$ two edge paths)

\end{lem}

Each edge of $K_{n,n}$ is present in $2n - 1$ geodesic paths, hence
\begin{equation*}
    (2n - 1) \sum_{i=1}^na_i = \frac{(n^3)(n^3+1)}{2}.
\end{equation*}
The sum of edge weights is always a natural number, hence this equality is satisfied only if $n = 1,2$ or $5$. The first one is $K_2$ and the second one is $C_4$, both of whom are already shown to be geodesic Leech. The computation shows that the only other possible geodesic Leech regular complete bipartite graph is $K_{5,5}$.
\begin{thm}
 Regular Complete Bipartite Graphs $K_{n,n}$ are not Geodesic Leech for $n \neq 1,2,5$
\end{thm}
\begin{lem}
    The edge transitive cubic graph on 6 vertices is not Geodesic Leech.
    \begin{proof}
        Let G be the edge transitive cubic graph on 6 vertices,
        $t_{gp}(G) = 27 $ (9 edges and 18 two-edge geodesic paths.)
        Each edge is present in 4 two-edge geodesic paths, which implies
        \begin{equation*}
            5 \sum_{i=1}^9a_i = \frac{27 \times 28}{2} = 378.
        \end{equation*}
        This is impossible since $\sum_{i=1}^na_i$ is always a natural number.
    \end{proof}

\end{lem}
\vspace{1.5cm}
The following labeling (See Figure \ref{Figure 7}) shows that the non edge transitive cubic graph on 6 vertices is a Geodesic Leech graph. \\

\begin{figure}[h]\center \label{fig7}
  \includegraphics[width=6.5cm]{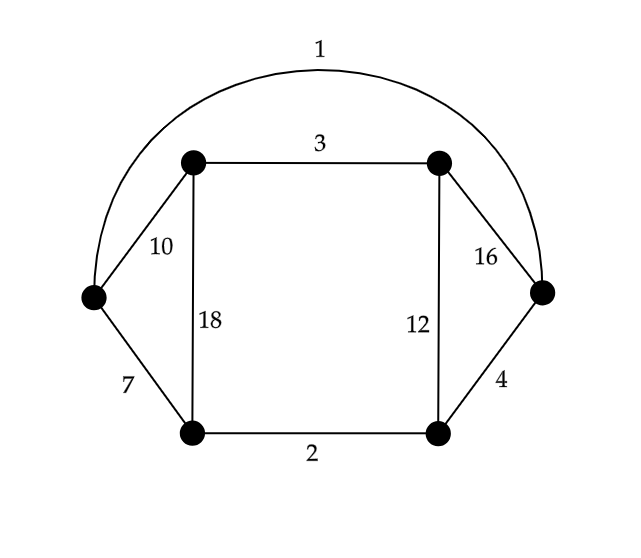}\\[-1cm]
  \caption{Geodesic Leech Cubic graph on 6 vertices}\label{Figure 7}
\end{figure}

   \begin{thm}
       Degree sequence does not characterize Geodesic Leech Graphs.
   \end{thm}

\section{Concluding Remarks and Open Problems}

In this paper, we have proved that the cycles $C_n$ for $n \geq 5$ are non-geodesic Leech graphs. We also proved that there are at most three regular complete bipartite graphs that are geodesic Leech. Though the results are not positive with respect to application point of view, they settle one of the problems suggested in \cite{See2}. We also showed that degree sequence does not characterize geodesic Leech graphs.\\

The wheel graph on $n$ vertices, $W_n$ is the cycle $C_{n-1}$ together with a vertex adjacent to all the vertices of $C_{n-1}$. In our search for more geodesic Leech graphs, we have obtained geodesic Leech labeling of the wheel graphs $W_5$ and $W_6$ (See Figure \ref{Figure 1}). But, we think that $W_n$ for $n \geq 7$ are non-geodesic Leech graphs. To check whether a wheel graph is a geodesic Leech graph or not, what we need first is the geodesic path number of $W_n$. The following theorem gives the geodesic path number of $W_n$ for $n \geq 5$.

\begin{thm}
  The geodesic path number of the wheel graph $W_n$ is $t_{gp}(W_n) =\frac{(n-1)(n+2)}{2}$.
\end{thm}
\begin{proof}
  The diameter of $W_n$ for $n \geq 5$ is two. The number of geodesic paths of length 1 is the number of edges in $W_n$ which is $2(n-1)$. The number of geodesic paths of length two completely contained in $C_{n-1}$ is $n-1$. The number of geodesic paths of length two with the universal vertex as the central vertex is $\frac{(n-1)(n-4)}{2}$. Therefore, $t_{gp}(W_n) = 3(n-1)+\frac{(n-1)(n-4)}{2}=\frac{n^2+n-2}{2} =\frac{(n-1)(n+2)}{2}$ .
\end{proof}

\begin{figure}[h]\center \label{fig1}
  \includegraphics[width=6.5cm]{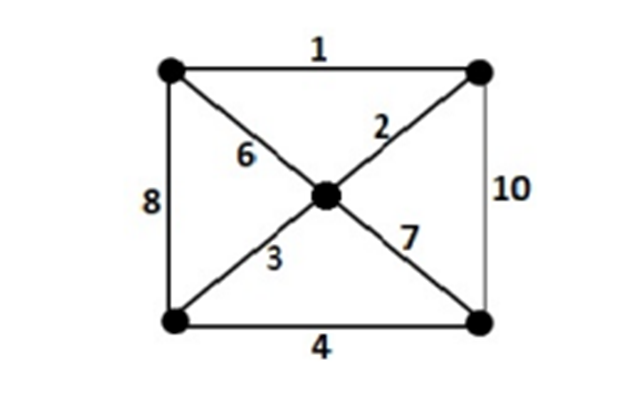}
  \includegraphics[width=6.5cm]{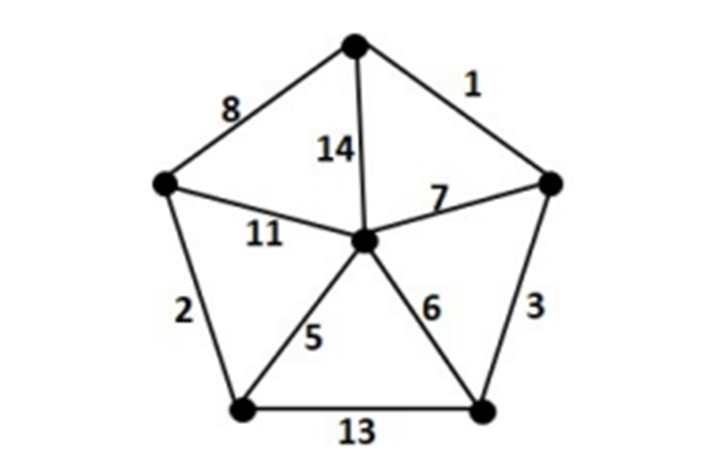}\\[0.3cm]
  \caption{Geodesic Leech labeling of $W_5$ and $W_6$}\label{Figure 1}
\end{figure}

$W_7$ is an almost geodesic Leech graph from the following labeling (Figure \ref{Figure 6}), but whether or not it is Geodesic Leech remains to be seen.

\begin{figure}[h] \label{fig6}
    \centering
    \includegraphics[width=6cm]{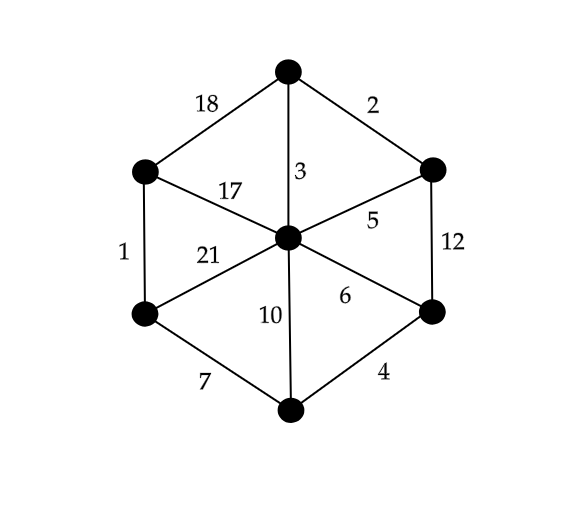}\\[-0.6cm]
    \caption{Almost Geodesic Leech Labeling of $W_7$}
    \label{Figure 6}
\end{figure}

\ni{\bf Problem:} The wheel graph $W_n$ for $n \geq 7$ is not a geodesic Leech graph.\\

\ni{\bf Acknowledgement:} The first author thank Cochin University of Science and Technology for granting project under Seed Money for New Research Initiatives (order No.CUSAT/PL(UGC).A1/1112/2021) dated 09.03.2021. The second author would like to place on record, his gratitude towards Council of Scientific \& Industrial Research (CSIR) for the financial assistance received (File No. 09/0239(17181)/2023-EMR-I).\\

{}
\newpage
\begin{center}
  \Large{\bf APPENDIX I}
  \end{center}
\begin{lstlisting}[language=Python]
def commonelement(x,y):
   common=0
   for value in x:
      if value in y:
         common=1
   return (common)

A=[]
B=[]
C=[]
A.append(1)
for i in range(2,32):
    A.append(i)
    B.append(A[0]+A[1])
    for j in range(2,32):
        if j not in A:
            if j not in B:
                A.append(j)
                if A[0]+A[1]+A[2] > 50:
                    print("If the edge labels are cyclically", A, "then the path weight exceeds 50")
                    A.pop()
                    break
                C = [A[1]+A[2],A[0]+A[1]+A[2]]
                if commonelement(A,C)==0:
                   if commonelement(B,C)==0:
                      B.extend(C)
                   else:
                      print("If the edge labels are cyclically", A, "then a path weight is repeated")
                      A.pop()
                      C.clear()
                      continue
                else:
                      print("If the edge labels are cyclically", A, "then a path weight is repeated")
                      A.pop()
                      C.clear()
                      continue
                C.clear()
            else:
                continue
        else:
            continue
        for k in range(2,32):
            if k not in A:
                if k not in B:
                    A.append(k)
                    if A[0]+A[1]+A[2]+A[3] > 50:
                        print("If the edge labels are cyclically", A, "then the path weight exceeds 50")
                        A.pop()
                        break
                    C=[A[2]+A[3],A[1]+A[2]+A[3],A[0]+A[1]+A[2]+A[3]]
                    if commonelement(A,C)==0:
                      if commonelement(B,C)==0:
                         B.extend(C)
                      else:
                         print("If the edge labels are cyclically", A, "then a path weight is repeated")
                         A.pop()
                         C.clear()
                         continue
                    else:
                       print("If the edge labels are cyclically", A, "then a path weight is repeated")
                       A.pop()
                       C.clear()
                       continue
                    C.clear()
                else:
                    continue
            else:
                continue
            for l in range(2,32):
                if l not in A:
                    if l not in B:
                        A.append(l)
                        if A[0]+A[1]+A[2]+A[3]+A[4] > 50:
                            print("If the edge labels are cyclically", A, "then the path weight exceeds 50")
                            A.pop()
                            break
                        C=[A[3]+A[4],A[2]+A[3]+A[4],A[1]+A[2]+A[3]+A[4],A[0]+A[1]+A[2]+A[3]+A[4]]
                        if commonelement(A,C)==0:
                           if commonelement(B,C)==0:
                              B.extend(C)
                           else:
                              print("If the edge labels are cyclically", A, "then a path weight is repeated")
                              A.pop()
                              C.clear()
                              continue
                        else:
                           print("If the edge labels are cyclically", A, "then a path weight is repeated")
                           A.pop()
                           C.clear()
                           continue
                        C.clear()
                    else:
                        continue
                else:
                    continue
                for m in range(2,32):
                    if m not in A:
                        if m not in B:
                            A.append(m)
                            if A[1]+A[2]+A[3]+A[4]+A[5] > 50:
                                print("If the edge labels are cyclically", A, "then the path weight exceeds 50")
                                A.pop()
                                break
                            C=[A[4]+A[5],A[3]+A[4]+A[5],A[2]+A[3]+A[4]+A[5],A[1]+A[2]+A[3]+A[4]+A[5]]
                            if commonelement(A,C)==0:
                              if commonelement(B,C)==0:
                                 B.extend(C)
                              else:
                                 print("If the edge labels are cyclically", A, "then a path weight is repeated")
                                 A.pop()
                                 C.clear()
                                 continue
                            else:
                              print("If the edge labels are cyclically", A, "then a path weight is repeated")
                              A.pop()
                              C.clear()
                              continue
                            C.clear()
                        else:
                            continue
                    else:
                        continue
                    for n in range(2,32):
                        if n not in A:
                            if n not in B:
                                A.append(n)
                                if A[2]+A[3]+A[4]+A[5]+A[6] > 50:
                                    print("If the edge labels are cyclically", A, "then the path weight exceeds 50")
                                    A.pop()
                                    break
                                C=[A[5]+A[6],A[4]+A[5]+A[6],A[3]+A[4]+A[5]+A[6],A[2]+A[3]+A[4]+A[5]+A[6]]
                                if commonelement(A,C)==0:
                                 if commonelement(B,C)==0:
                                    B.extend(C)
                                 else:
                                    print("If the edge labels are cyclically", A, "then a path weight is repeated")
                                    A.pop()
                                    C.clear()
                                    continue
                                else:
                                 print("If the edge labels are cyclically", A, "then a path weight is repeated")
                                 A.pop()
                                 C.clear()
                                 continue
                                C.clear()
                            else:
                                continue
                        else:
                            continue
                        for o in range(2,32):
                            if o not in A:
                                if o not in B:
                                    A.append(o)
                                    if A[3]+A[4]+A[5]+A[6]+A[7] > 50:
                                        print("If the edge labels are cyclically", A, "then the path weight exceeds 50")
                                        A.pop()
                                        break
                                    C=[A[6]+A[7],A[5]+A[6]+A[7],A[4]+A[5]+A[6]+A[7],A[3]+A[4]+A[5]+A[6]+A[7]]
                                    if commonelement(A,C)==0:
                                       if commonelement(B,C)==0:
                                          B.extend(C)
                                       else:
                                          print("If the edge labels are cyclically", A, "then a path weight is repeated")
                                          A.pop()
                                          C.clear()
                                          continue
                                    else:
                                       print("If the edge labels are cyclically", A, "then a path weight is repeated")
                                       A.pop()
                                       C.clear()
                                       continue
                                    C.clear()
                                else:
                                    continue
                            else:
                                continue
                            for p in range(2,32):
                                if p not in A:
                                    if p not in B:
                                        A.append(p)
                                        if A[4]+A[5]+A[6]+A[7]+A[8] > 50:
                                            print("If the edge labels are cyclically", A, "then the path weight exceeds 50")
                                            A.pop()
                                            break
                                        C=[A[7]+A[8],A[6]+A[7]+A[8],A[5]+A[6]+A[7]+A[8],A[4]+A[5]+A[6]+A[7]+A[8]]
                                        if commonelement(A,C)==0:
                                          if commonelement(B,C)==0:
                                             B.extend(C)
                                          else:
                                             print("If the edge labels are cyclically", A, "then a path weight is repeated")
                                             A.pop()
                                             C.clear()
                                             continue
                                        else:
                                          print("If the edge labels are cyclically", A, "then a path weight is repeated")
                                          A.pop()
                                          C.clear()
                                          continue
                                        C.clear()
                                    else:
                                        continue
                                else:
                                    continue
                                q=85-(A[0]+A[1]+A[2]+A[3]+A[4]+A[5]+A[6]+A[7]+A[8])
                                if q < 2:
                                    print("If the edge labels are cyclically", A, "then the last label must be negative to get the total sum as 85; which is a contradiction")
                                    A.pop()
                                    del B[-4:]
                                    break
                                else:
                                    if q not in A:
                                        if q not in B:
                                            A.append(q)
                                            if A[5]+A[6]+A[7]+A[8]+A[9] > 50:
                                                print("If the edge labels are cyclically", A, "then the path weight exceeds 50")
                                                del A[-2:]
                                                del B[-4:]
                                                break
                                            C=[A[8]+A[9], A[7]+A[8]+A[9], A[6]+A[7]+A[8]+A[9], A[5]+A[6]+A[7]+A[8]+A[9], A[9]+A[0], A[9]+A[0]+A[1], A[9]+A[0]+A[1]+A[2], A[9]+A[0]+A[1]+A[2]+A[3],A[8]+A[9]+A[0],A[8]+A[9]+A[0]+A[1],A[8]+A[9]+A[0]+A[1]+A[2],A[7]+A[8]+A[9]+A[0],A[7]+A[8]+A[9]+A[0]+A[1],A[6]+A[7]+A[8]+A[9]+A[0]]
                                            if commonelement(B,C)==0:
                                               if commonelement(B,C)==0:
                                                    B.extend(C)
                                               else:
                                                    print("If the edge labels are cyclically", A, "then a path weight is repeated")
                                                    del A[-2:]
                                                    del B[-4:]
                                                    continue
                                            else:
                                               print("If the edge labels are cyclically", A, "then a path weight is repeated")
                                               del A[-2:]
                                               del B[-4:]
                                               continue
                                            if max(B)==50:
                                                print("Required labeling is ",A)
                                                exit()
                                            else:
                                                del A[-2:]
                                                del B[-14:]
                                                continue
                                        else:
                                            print("If the edge labels are cyclically", A, "then the last label must be", q, "which is a repeatation")
                                            A.pop()
                                            del B[-4:]
                                            continue
                                    else:
                                        print("If the edge labels are cyclically", A, "then the last label must be", q, "which is a repeatation")
                                        A.pop()
                                        del B[-4:]
                                        continue
                            A.pop()
                            del B[-4:]
                        A.pop()
                        del B[-4:]
                    A.pop()
                    del B[-4:]
                A.pop()
                del B[-4:]
            A.pop()
            del B[-3:]
        A.pop()
        del B[-2:]
    A.pop()
    B.pop()
print("No such labeling exist")
    

                                
                                    
                                    

        
\end{lstlisting}
\newpage
\begin{center}
  \Large{\bf APPENDIX II}
  \end{center}

\begin{figure}[h]\center \label{fig4}
  \includegraphics[width=3.5cm]{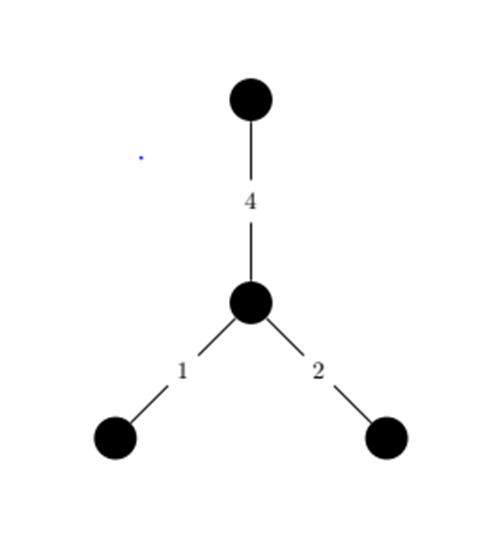}
  \includegraphics[width=6cm]{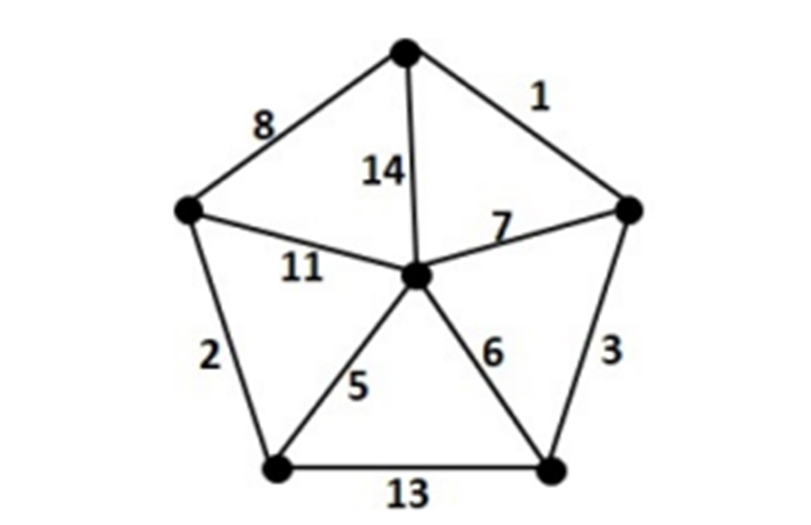}
  \includegraphics[width=3.2cm]{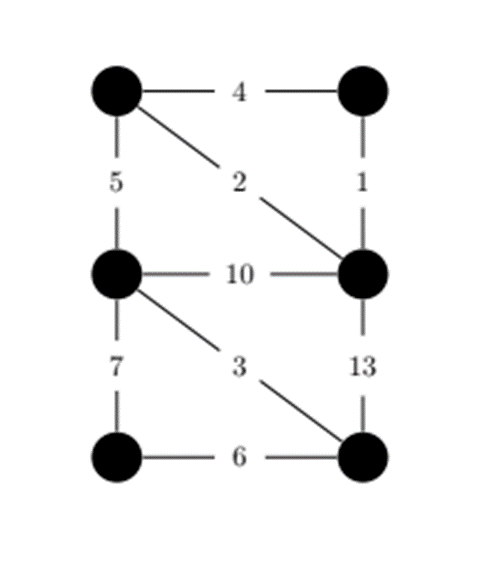}
  \includegraphics[width=6.75cm]{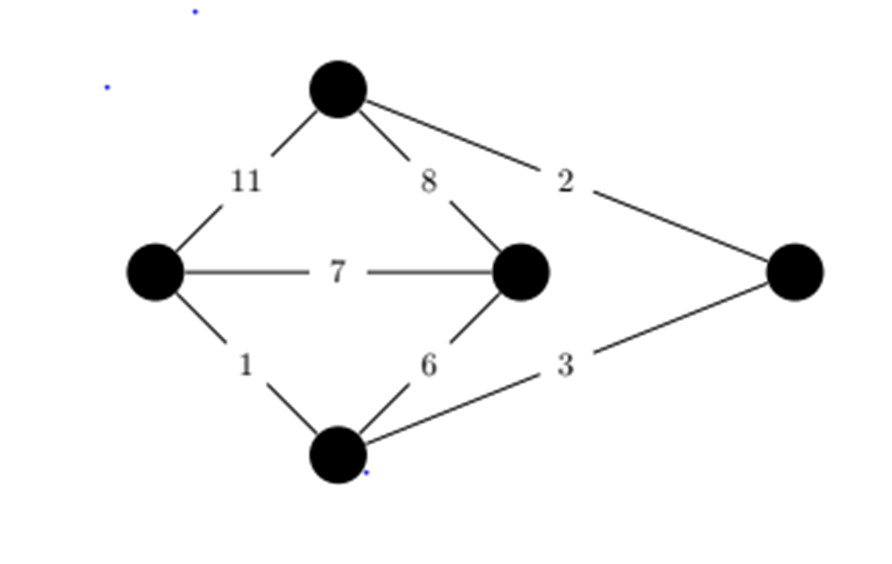}
  \includegraphics[width=6cm]{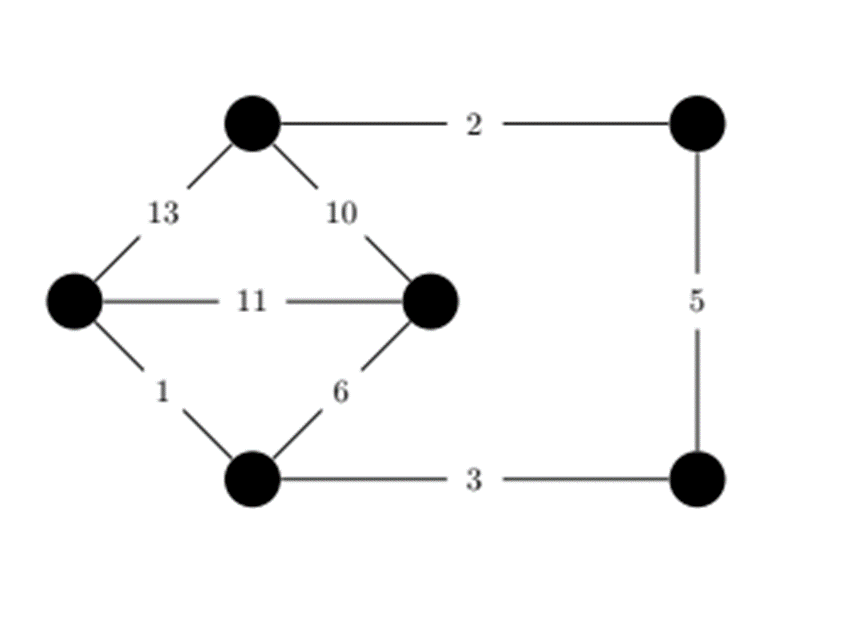}
  \includegraphics[width=6cm]{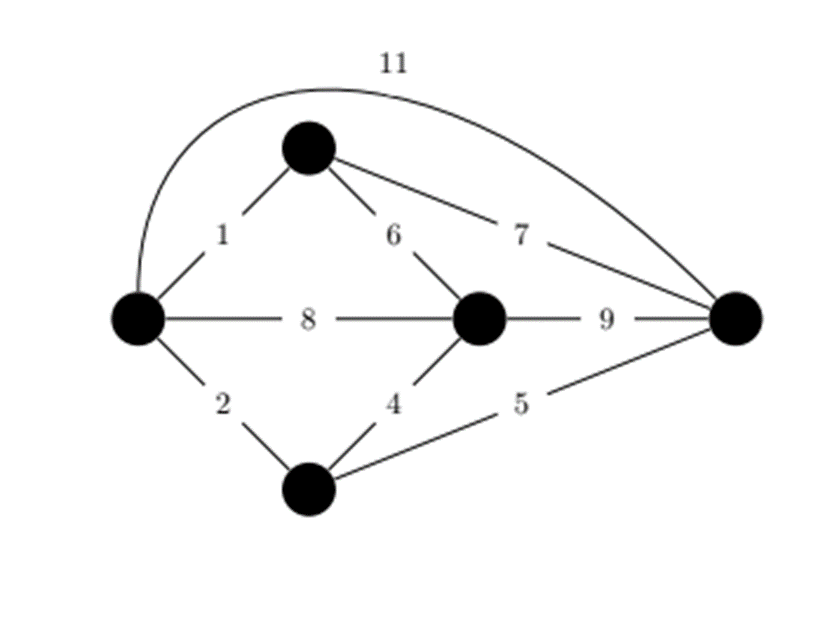}
  \includegraphics[width=4cm]{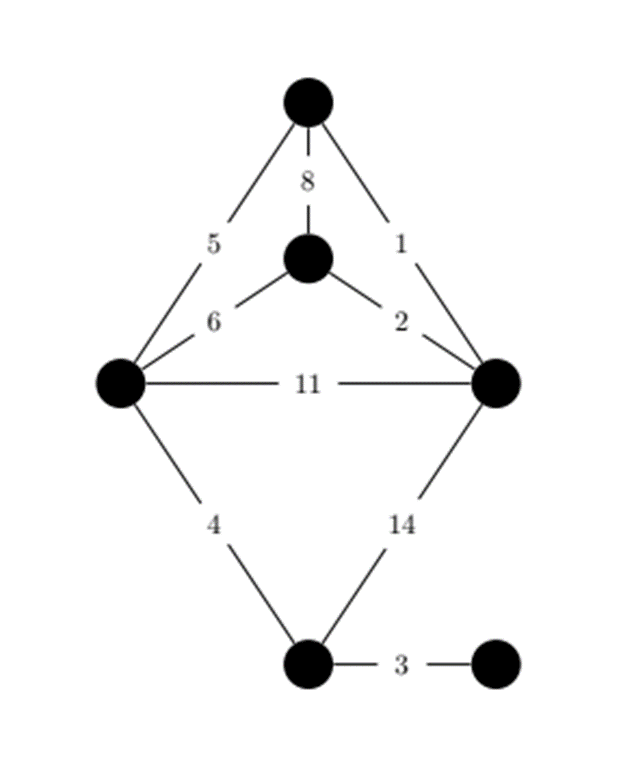}
  \includegraphics[width=4cm]{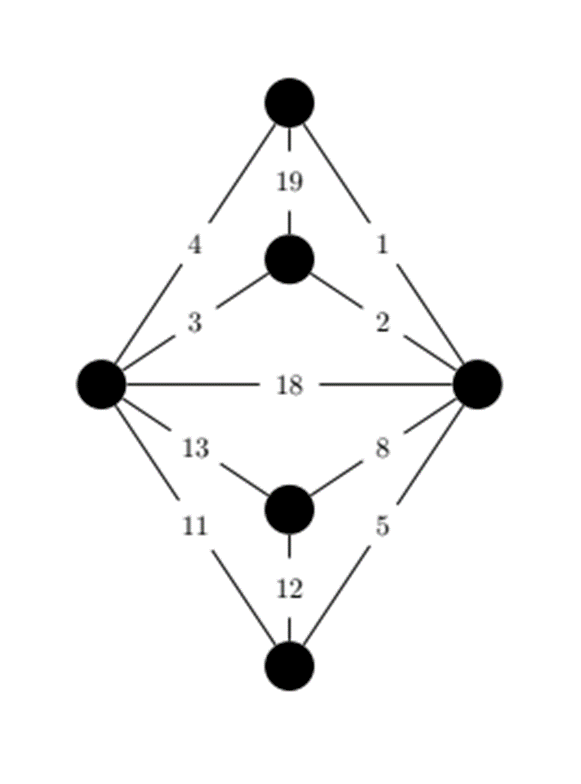}
  \caption{Geodesic Leech labeling of 8 Forbidden Subgraphs of Linegraphs}\label{Figure 4}
\end{figure}

The one forbidden subgraph of linegraphs which is not in the above list was verified using a computer program to be non geodesic Leech. But we observed that it is almost geodesic Leech\cite{See2} and found two distinct labelings for the same (See Figure \ref{Figure 5}).

\begin{figure}[ht]\center \label{fig5}
\includegraphics[width=6.5cm]{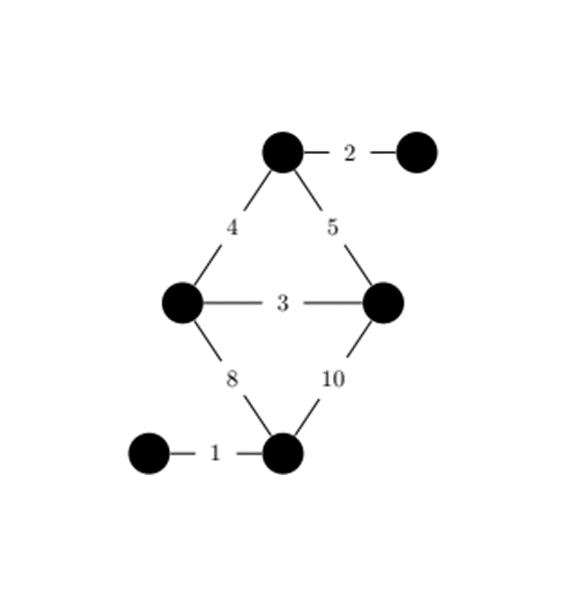}
  \includegraphics[width=6cm]{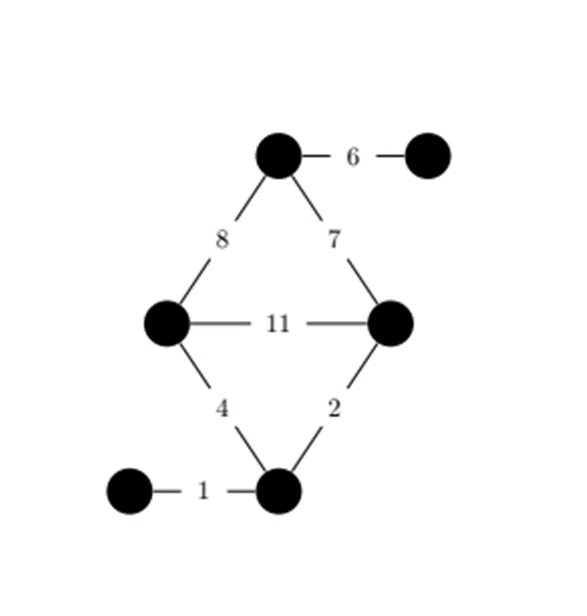}
  \caption{An Almost Geodesic Leech Graph with two distinct Labelings}\label{Figure 5}
\end{figure}
\vspace{2.5cm}
\subsection*{\centering\textbf{\underline{Some Geodesic Leech Graphs}}}
\vspace{1cm}

\tikzset{every picture/.style={line width=0.75pt}} 


\vspace{2cm}
\subsection*{\centering\textbf{\underline{Almost Geodesic Leech Labeling of Some Graphs}}}
\vspace{2cm}

\tikzset{every picture/.style={line width=0.75pt}} 

%

\end{document}